\documentclass[leqno,11pt]{amsart}
\usepackage{}
\usepackage{amssymb, amsmath}
\usepackage{amsthm, amsfonts,mathrsfs}
\def\inte#1{
\displaystyle\mathop{#1\kern0pt}^\circ }

\def\virgp{\raise 2pt\hbox{,}}
\def\cdotpv{\raise 2pt\hbox{;}}

\def\C{\mathop{\mathbb C\kern 0pt}\nolimits}
\def\DD{\mathop{\mathbb D\kern 0pt}\nolimits}
\def\EE{\mathop{{\mathbb E \kern 0pt}}\nolimits}
\def\K{\mathop{\mathbb K\kern 0pt}\nolimits}
\def\N{\mathop{\mathbb N\kern 0pt}\nolimits}
\def\Q{\mathop{\mathbb Q\kern 0pt}\nolimits}
\def\R{\mathop{\mathbb R\kern 0pt}\nolimits}
\def\SS{\mathop{\mathbb S\kern 0pt}\nolimits}
\def\ZZ{\mathop{\mathbb Z\kern 0pt}\nolimits}
\def\TT{\mathop{\mathbb T\kern 0pt}\nolimits}
\def\P{\mathop{\mathbb P\kern 0pt}\nolimits}

\newcommand{\beq}{\begin{equation}}
\newcommand{\eeq}{\end{equation}}
\newcommand{\ben}{\begin{eqnarray}}
\newcommand{\een}{\end{eqnarray}}
\newcommand{\beno}{\begin{eqnarray*}}
\newcommand{\eeno}{\end{eqnarray*}}

\newtheorem{thm}{Theorem}[section]
\newtheorem{lem}{Lemma}[section]
\newtheorem{rmk}{Remark}[section]

\newtheorem{prop}{Proposition}[section]
\renewcommand{\theequation}{\thesection.\arabic{equation}}

\begin{document}

\title[]
{Notes on Liouville type Theorems for the stationary compressible Navier-Stokes equations}

\author[Z. Li and P. Niu]{Zhouyu Li\textsuperscript{1}      \and   Pengcheng Niu\textsuperscript{1, *}}

\thanks{$^*$ Corresponding author}
\thanks{$^1$ School of Mathematics and Statistics, Northwestern Polytechnical University, Xi'an 710129, China}
\thanks{E-mail address: zylimath@163.com (Z. Li); pengchengniu@nwpu.edu.cn (P. Niu)}

\begin{abstract}
In this paper, we investigate the three dimensional stationary compressible Navier-Stokes equations, and obtain Liouville type theorems if a smooth solution $(\rho, \mathbf{u})$ satisfies some suitable conditions. In particular, our results improve and generalize the corresponding result of Li and Yu (2014) \cite{L2014}.

\end{abstract}


\date{}

\maketitle



\vskip 0.2cm


\renewcommand{\theequation}{\thesection.\arabic{equation}}
\setcounter{equation}{0}
\section{Introduction}
This paper is concerned with the following stationary barotropic compressible Navier-Stokes equations on $\mathbb{R}^3$
\begin{equation}\label{NS}
    \begin{cases}
    \operatorname{div}(\rho \mathbf{u})=0, \\
    \operatorname{div}(\rho \mathbf{u}\otimes \mathbf{u})- \nu\Delta \mathbf{u}-(\lambda + \nu)\nabla \operatorname{div} \mathbf{u} +\nabla P =0, \\
    \end{cases}
\end{equation}
where the vector $\mathbf{u}$ denotes the flow velocity field and the scalar function $\rho$ represents the density of the fluid.
The pressure $P$ is given by the $\gamma-$law:
$$P(\rho)=a\rho^\gamma,$$
where $a>0$ and $\gamma>1$ are physical constants. The shear viscosity $\nu$ and bulk viscosity $\lambda$ are both constants
and satisfy
$$\nu>0\ , \ \lambda+\frac{2}{3}\nu>0.$$
For more physical explanations about \eqref{NS}, see \cite{F2004, L1998}.

In the past decades, the time-dependent compressible Navier-Stokes equations have been studied by many authors. See \cite{D2000, X1998} for global existence and blow-up criteria of solutions, and \cite{D1997, Q2009} for regularity criterion of weak solutions. In general, Liouville type theorems appear naturally when studying the regularity of the time-dependent Navier-Stokes equations. However, Liouville type theorems
have not been solved yet, as far as we know, even in the stationary case.

Here we consider the Liouville type theorems for the stationary compressible Navier-Stokes equations \eqref{NS}.
In \cite{Chae2012}, Chae showed the Liouville type theorems for the compressible Navier-Stokes equations on $\mathbb{R}^n$. In particular, he stated that if the smooth solution $(\rho, \mathbf{u})$
satisfies
\begin{equation}\label{1-2}
     \begin{split}
\|\rho\|_{L^\infty(\mathbb{R}^3)}+\|\nabla \mathbf{u}\|_{L^2(\mathbb{R}^3)}+\|\mathbf{u}\|_{L^\frac{3}{2}(\mathbb{R}^3)}<\infty,
     \end{split}
\end{equation}
then $u\equiv 0$ and $\rho=constant$. Later, Li and Yu in \cite{L2014} pointed out that the condition
\begin{equation*}\label{1-2a}
     \begin{split}
\|\rho\|_{L^\infty(\mathbb{R}^3)}+\|\nabla \mathbf{u}\|_{L^2(\mathbb{R}^3)}<\infty
     \end{split}
\end{equation*}
is very natural because most physical flows have bounded density and finite enstrophy, and improved the result of Chae, assuming that
\begin{equation}\label{1-3}
     \begin{split}
\|\rho\|_{L^\infty(\mathbb{R}^3)}+\|\nabla \mathbf{u}\|_{L^2(\mathbb{R}^3)}+\|\mathbf{u}\|_{L^\frac{9}{2}(\mathbb{R}^3)}<\infty.
     \end{split}
\end{equation}
The interested readers can also refer \cite{Li2020}, which proved a different result.
We note that the condition \eqref{1-3} is weaker than \eqref{1-2} in the sense that $\mathbf{u}$ can decay more slowly at infinity.
Thus, a natural and interesting problem is how to further weaken the integrability condition $\mathbf{u}\in L^{\frac{9}{2}}(\mathbb{R}^3)$. The purpose of this paper is to give a positive answer.
We state the first result of this paper as follows:
\begin{thm}\label{thm-1}
Suppose that $(\rho, \mathbf{u})$ is a smooth solution to \eqref{NS} with $\rho\in L^\infty (\mathbb{R}^3)$, $\nabla \mathbf{u}\in L^2(\mathbb{R}^3)$
and $\mathbf{u}\in L^{p, q}(\mathbb{R}^3)$ for $3< p <\frac{9}{2}$, $3\leq q\leq \infty$ or $p=q=3$. Then $\mathbf{u}\equiv 0$ and $\rho=constant$ on $\mathbb{R}^{3}$.
\end{thm}

Now one can ask what happens for the values $ p\geq \frac{9}{2}$ $?$ We apply some ideas of the papers \cite{K2017, Wang2018} dealing with the Liouville type theorems for stationary incompressible Navier-Stokes equations. Let us define
$$M_{p, q}(R):= R^{\frac{2}{3}-\frac{3}{p}}\|\mathbf{u}\|_{L^{p, q}(R\leq |x|\leq 2R)}.$$ Our second result is
\begin{thm}\label{thm-2}
Suppose that $(\rho, \mathbf{u})$ is a smooth solution to \eqref{NS} with $\rho\in L^\infty (\mathbb{R}^3)$
and $\nabla \mathbf{u}\in L^2(\mathbb{R}^3).$  For $p\geq \frac{9}{2}$, $3\leq q\leq \infty$, assume that
\begin{equation*}
     \begin{split}
\liminf_{R \rightarrow \infty} M_{\,p, \,q}(R)<\infty,
     \end{split}
\end{equation*}
then
\begin{equation}\label{1-1a}
     \begin{split}
D(u):=\int_{\mathbb{R}^3}|\nabla \mathbf{u}|^2 \, dx \leq C_0\liminf_{R \rightarrow \infty}M^3_{\,p, \,q}(R).
     \end{split}
\end{equation}
If moreover assume
\begin{equation}\label{1-1}
     \begin{split}
\liminf_{R \rightarrow \infty}M^3_{\,p, \,q}(R)\leq \delta D(\mathbf{u})
\end{split}
     \end{equation}
for some $0<\delta<1/C_0,$ then $\mathbf{u}\equiv 0$ and $\rho=constant$ on $\mathbb{R}^{3}$.
\end{thm}

Indeed, we establish the Liouville type theorems in the setting of Lorentz spaces, which can be regarded as a natural generalisation of Lesbesgue spaces. Roughy speaking, we prove that a  smooth solution $(\rho, \mathbf{u})$ is trivial if $\rho\in L^\infty (\mathbb{R}^3)$, $\nabla \mathbf{u}\in L^2(\mathbb{R}^3)$ and $\mathbf{u}\in L^{p, q}$ with $ p \geq 3$. Compared with the case in \cite{K2017, Wang2018}, we investigate the Liouville type theorems for the stationary compressible Navier-Stokes equations \eqref{NS}. Note that in \eqref{NS}, we need the condition
$\nabla \mathbf{u}\in L^2(\mathbb{R}^3)$ in order to deal with the pressure term. On the other hand, we more carefully discuss $\mathbf{u}\in L^{p, q}$ for the two cases: $3\leq p < \frac{9}{2}$ and $p\geq\frac{9}{2}$ in contrast to the result given in \cite{K2017, Wang2018}.

\begin{rmk}
{\rm{(i)}}. It should be noted that if we consider the integrability condition $\mathbf{u}\in L^{p, q}(\mathbb{R}^3)$ with $3< p <\frac{9}{2}$, $3\leq q \leq \infty$ or $p=q=3$ in Theorem \ref{thm-1}, then we do not need any additional assumption which is similar to \eqref{1-1}.\\
{\rm{(ii)}}. For the value $p=\frac{9}{2}$ and $q=\infty$ in Theorem \ref{thm-2},
our result can be regarded as a relaxation to the corresponding result of Li and Yu \cite{L2014}.
\end{rmk}

The organization of this paper is as follows: In Section 2, we collect some elementary facts. Section 3 is devoted to obtaining a prior estimate, which is the key of our proof.
 Finally, the proofs of Theorem \ref{thm-1} and \ref{thm-2} are given in Section 4.


\renewcommand{\theequation}{\thesection.\arabic{equation}}
\setcounter{equation}{0} 

\section{Preliminaries}
We recall the definition of Lorentz space. Given $1\leq p< \infty , 1\leq q\leq \infty$, we say that
a measurable function $f \in L^{p, q}(\mathbb{R}^3)$ if $\|f\|_{L^{p, q}(\mathbb{R}^3)}< \infty,$
where
\begin{equation*}\label{2-1a}
\|f\|_{L^{p, q}(\mathbb{R}^3)}:=
    \begin{cases}
        \left(\int_0^{\infty}t^{q-1}|\{x\in \mathbb{R}^3: |f(x)|> t\}|^{\frac{q}{p}}\, dt\right)^{\frac{1}{q}}, \quad \mathrm{if}\, \quad q<+\infty, \\
        \sup\limits_{t>0} t|\{x\in \mathbb{R}^3: |f(x)|> t\}|^{\frac{1}{p}}, \quad \mathrm{if}\,  \quad q=+\infty.
    \end{cases}
\end{equation*}
The space satisfies the continuous embeddings
\begin{equation*}
  L^{p}(\mathbb{R}^3)=L^{p, p}(\mathbb{R}^3)\hookrightarrow L^{p, q}(\mathbb{R}^3)\hookrightarrow L^{p, \infty}(\mathbb{R}^3) \ , \  p\leq q <\infty.
\end{equation*}
It should be stressed that $\|\cdot\|_{L^{p, q}}$ is a quasi-norm, namely, $\|\cdot\|_{L^{p, q}}$ do not satisfy the usual triangle inequality. Instead, we have
$$\|f+g\|_{L^{p, q}}\leq C(p, q) (\|f\|_{L^{p, q}}+\|g\|_{L^{p, q}}),$$
with $C(p,q)=2^{1/p}\max (1, 2^{(1-q)/ q})$.
See \cite{L1950} for details.

The following inequalities in Lorentz spaces are useful.

\begin{lem}[H\"{o}lder inequality, \cite{O1963}]\label{Lem2-1}
\label{lem}
Let $f\in L^{p_1,q_1}(\mathbb{R}^3)$ and $g\in L^{p_2,q_2}(\mathbb{R}^3)$ with $1\leq p_1, p_2\leq\infty$, $1\leq q_1, q_2\leq\infty.$
Then $fg\in L^{p,q}(\mathbb{R}^3)$ with $\frac{1}{p}=\frac{1}{p_1}+\frac{1}{p_2}$, $\frac{1}{q}\leq\frac{1}{q_1}+\frac{1}{q_2}$
and
\begin{equation*}\label{2-1}
     \begin{split}
   \|fg\|_{L^{p,q}(\mathbb{R}^3)}\leq C\|f\|_{L^{p_1,q_1}(\mathbb{R}^3)}\|g\|_{L^{p_2,q_2}(\mathbb{R}^3)}
     \end{split}
\end{equation*}
for a constant $C>0.$
\end{lem}

\begin{lem}[$\operatorname{Calder\acute{o}n}$-Zygmund inequality, \cite{B1976}]\label{Lem2-2}
Let $\Omega$ be a bounded domain in $\mathbb{R}^n$, $1<p<\infty$, $1<q\leq\infty$ and $f\in L^{p, q}(\Omega)$.
Then
\begin{equation*}\label{2-2}
     \begin{split}
   \|\nabla^2(-\Delta)^{-1}f\|_{L^{p,q}(\Omega)}\leq C\|f\|_{L^{p,q}(\Omega)},
     \end{split}
\end{equation*}
where the constant $C>0$ is independent of $\Omega$.
\end{lem}


\renewcommand{\theequation}{\thesection.\arabic{equation}}
\setcounter{equation}{0} 

\section{A priori estimate}
In this section, we derive a local estimate of $\nabla \mathbf{u}$ by means of $\mathbf{u}$:
\begin{prop}\label{prop3-1}
Let $(\rho, \mathbf{u})$ be a smooth solution to \eqref{NS} with $\rho\in L^\infty (\mathbb{R}^3)$ and $\nabla u\in L^2(\mathbb{R}^3)$.
If $p > 3 $, $3\leq q\leq \infty$ or $p=q=3$,
then we have
\begin{equation}\label{3-0a}
     \begin{split}
     \int_{|x|\leq R}|\nabla \mathbf{u}|^2 \, dx + \int_{|x|\leq R}|\operatorname{div} \mathbf{u}|^2 \, dx \leq C_0 D_1,
     \end{split}
\end{equation}
where $D_1:=R^{1-\frac{6}{p}}\|\mathbf{u}\|^2_{L^{p, q}(R\leq|x|\leq2R)}+ R^{2-\frac{9}{p}}\|\mathbf{u}\|^3_{L^{p, q}(R\leq|x|\leq2R)}+ R^{\frac{1}{2}-\frac{3}{p}}\|\mathbf{u}\|_{L^{p, q}(R\leq|x|\leq2R)}$ and the constant $C_0$ is independent of $R>0$.

\end{prop}

\begin{proof}
Let $\varphi \in C_0^\infty (\mathbb{R}^3)$ be a radial cut-off function satisfying
\begin{equation*}\label{3-1}
\varphi(|x|)=
    \begin{cases}
        1, \quad \mathrm{if}\,  |x|< 1, \\
        0, \quad \mathrm{if}\,  |x|> 2,
    \end{cases}
\end{equation*}
and $0 \leq \varphi (|x|) \leq 1 $ for $1\leq |x|\leq2.$ For each given $R>0$, we define
$\varphi_R(x):=\varphi(\frac{|x|}{R})$ satisfying
$$\|\nabla^k \varphi_R\|_{L^\infty} \leqslant CR^{-k}$$
for $k=0, 1, 2$ with some positive constant $C$ independent of $x\in \mathbb{R}^3.$

Taking the inner product of  $\eqref{NS}_2$ with $\mathbf{u}\varphi_R^2$ and integrating by parts over $\mathbb{R}^3$,
it follows
\begin{equation*}\label{3-2}
     \begin{split}
   &\nu\int_{\mathbb{R}^3}\nabla \mathbf{u} :\nabla(\mathbf{u}\varphi_R^2)\, dx + (\lambda+\nu)\int_{\mathbb{R}^3}\operatorname{div} \mathbf{u} \operatorname{div} (\mathbf{u}\varphi_R^2)\, dx
   +\int_{\mathbb{R}^3} \operatorname{div} (\rho \mathbf{u}\otimes \mathbf{u})\cdot \mathbf{u}\varphi_R^2\, dx\\
   &+ \int_{\mathbb{R}^3}\nabla P\cdot \mathbf{u}\varphi_R^2\, dx=0.
     \end{split}
\end{equation*}
Applying the fact $\operatorname{div}(\rho \mathbf{u})=0$, we deduce that
\begin{equation}\label{3-3}
     \begin{split}
   &\nu\int_{\mathbb{R}^3}|\nabla \mathbf{u}|^2 \varphi_R^2\, dx + (\lambda+\nu)\int_{\mathbb{R}^3}|\operatorname{div} \mathbf{u}|^2 \varphi_R^2\, dx\\
   &=-2\nu\int_{\mathbb{R}^3}\varphi_R \nabla \mathbf{u} : \mathbf{u}\otimes \nabla \varphi_R \, dx
   -2(\lambda+\nu)\int_{\mathbb{R}^3} \varphi_R \operatorname{div} \mathbf{u} \mathbf{u}\cdot \nabla \varphi_R \, dx\\
   &\quad-\int_{\mathbb{R}^3}\rho \mathbf{u}\cdot \nabla \mathbf{u}\cdot \mathbf{u}\varphi_R^2 \, dx
   - \int_{\mathbb{R}^3}\nabla P\cdot \mathbf{u}\varphi_R^2\, dx\\
   &=\sum_{i=1}^4 I_i.
     \end{split}
\end{equation}

In the following, we estimate $I_i$ term by term. we assume $ p > 3$, $3\leq q\leq \infty$ or $p=q=3$.
For $I_1$, we get
\begin{equation*}
     \begin{split}
   I_1
   &=\nu\int_{\mathbb{R}^3} \mathbf{u}\cdot\operatorname{div}(\mathbf{u} \otimes \nabla (\varphi_R^2)) \, dx\\
   &=\nu\int_{\mathbb{R}^3} \mathbf{u}\cdot (\nabla \mathbf{u} \cdot \nabla (\varphi_R^2)+\mathbf{u} \Delta (\varphi_R^2)) \, dx\\
    &=\nu\int_{\mathbb{R}^3}|\mathbf{u}|^2 (\varphi_R\Delta \varphi_R+|\nabla \varphi_R|^2) \, dx.
     \end{split}
\end{equation*}
Using Lemma \ref{Lem2-1} implies that
\begin{equation}\label{3-4}
     \begin{split}
   |I_1|&\leq C (\nu) \int_{R\leq|x|\leq 2R}|\mathbf{u}|^2 (|\varphi_R\Delta \varphi_R|+|\nabla \varphi_R|^2) \, dx\\
   &\leq C  R^{-2} \||\mathbf{u}|^2\|_{L^{\frac{p}{2}, \frac{q}{2}}(R\leq|x|\leq 2R)}
   \|1\|_{L^{\frac{p}{p-2}, \frac{q}{q-2}}(R\leq|x|\leq 2R)}\\
   &\leq C  R^{1-\frac{6}{p}}\|\mathbf{u}\|^2_{L^{p, q}(R\leq|x|\leq 2R)}.
     \end{split}
\end{equation}

For $I_2$,  an application of Young inequality yields
\begin{equation}\label{3-5}
     \begin{split}
   |I_2|&\leq C(\lambda+\nu) \int_{R\leq|x|\leq 2R}| \varphi_R\operatorname{div} \mathbf{u}|| \mathbf{u}||\nabla \varphi_R| \, dx\\
   &\leq C(\lambda+\nu) R^{-1}\| \varphi_R\operatorname{div} \mathbf{u}\|_{L^{\frac{p}{p-1}, \frac{q}{q-1}}(R\leq|x|\leq 2R)}\|\mathbf{u}\|_{L^{p, q}(R\leq|x|\leq 2R)}\\
   &\leq C(\lambda+\nu) R^{-1}\| \varphi_R\operatorname{div} \mathbf{u}\|_{L^2(R\leq|x|\leq 2R)}\|\mathbf{u}\|_{L^{p, q}(R\leq|x|\leq 2R)}\|1\|_{L^{\frac{2p}{p-2}, \frac{2q}{q-2}}(R\leq|x|\leq 2R)}\\
   &\leq C(\lambda+\nu) R^{\frac{1}{2}-\frac{3}{p}}\| \varphi_R\operatorname{div} \mathbf{u}\|_{L^2(R\leq|x|\leq 2R)}\|\mathbf{u}\|_{L^{p, q}(R\leq|x|\leq 2R)}\\
   &\leq \frac{(\lambda+\nu)}{2} \| \varphi_R\operatorname{div} \mathbf{u}\|_{L^2(\mathbb{R}^3)}^2+CR^{1-\frac{6}{p}}\|\mathbf{u}\|^2_{L^{p, q}(R\leq|x|\leq 2R)}.\\
     \end{split}
\end{equation}

For $I_3$, we obtain from $\operatorname{div}(\rho \mathbf{u})=0$ that
\begin{equation*}
     \begin{split}
   I_3&=-\int_{\mathbb{R}^3} \rho \mathbf{u}\cdot \nabla \mathbf{u}\cdot \mathbf{u}\varphi_R^2 \, dx
   =-\frac{1}{2}\int_{\mathbb{R}^3} \rho \mathbf{u}\cdot \nabla |\mathbf{u}|^2 \varphi_R^2  \, dx
   =\frac{1}{2}\int_{\mathbb{R}^3} |\mathbf{u}|^2\operatorname{div}(\rho \mathbf{u}\varphi_R^2)  \, dx\\
  & =\frac{1}{2}\int_{\mathbb{R}^3} |\mathbf{u}|^2(\varphi_R^2\operatorname{div}(\rho \mathbf{u})+2\varphi_R\nabla \varphi_R\cdot\rho \mathbf{u})  \, dx
   =\int_{\mathbb{R}^3}|\mathbf{u}|^2 \varphi_R\nabla \varphi_R\cdot\rho \mathbf{u} \, dx,
     \end{split}
\end{equation*}
thus
\begin{equation}\label{3-6}
     \begin{split}
   |I_3|&\leq C R^{-1}\|\rho\|_{L^\infty}\||\mathbf{u}|^3\|_{L^{\frac{p}{3}, \frac{q}{3}}(R\leq|x|\leq 2R)}
   \|1\|_{L^{\frac{p}{p-3}, \frac{q}{q-3}}(R\leq|x|\leq 2R)}\\
   &\leq C R^{2-\frac{9}{p}}\|\mathbf{u}\|_{L^{p, q}(R\leq|x|\leq 2R)}^3.
     \end{split}
\end{equation}

To obtain the estimate for the pressure term, we recall the following lemma.
\begin{lem}[see \cite{L2014}]\label{Lem3-1}
\label{lem}
Let $P\in L^{\infty}(\mathbb{R}^3),$ $p_1\in L^{r_1}(\mathbb{R}^3),$ $p_2\in L^{r_2}(\mathbb{R}^3)$ with $1\leq r_1, r_2<\infty.$
Suppose that $P-p_1-p_2$ is weakly harmonic, that is
$$\Delta (P-p_1-p_2)=0$$
in the sense of distribution, then there exists a constant $c$ such that
\begin{equation*}
     \begin{split}
      P-p_1-p_2 = c \quad\, \mbox{a.e.}\quad\, x\in \mathbb{R}^3
      \end{split}
\end{equation*}
If furthermore $P(x)\geq 0$ a.e., then we also have $c\geq 0$.
\end{lem}

With Lemma \ref{Lem3-1} in hand, we give the estimate of $I_4$.
Taking the divergence on both sides of $\eqref{NS}_2$, we have
\begin{equation*}
     \begin{split}
   \Delta (P-p_1-p_2)=0,
     \end{split}
\end{equation*}
where $p_1:=(-\Delta)^{-1}\partial_i\partial_j (\rho u_iu_j)$ and $p_2:=(\lambda+2\nu)\operatorname{div}\mathbf{u}$.

Using the assumption $\nabla \mathbf{u}\in L^2{(\mathbb{R}^3)}$ and the Sobolev embedding $\dot{H}^1(\mathbb{R}^3)\hookrightarrow L^6(\mathbb{R}^3)$, it follows
$$p_1\in L^3(\mathbb{R}^3)\ , \ p_2\in L^2(\mathbb{R}^3).$$
Due to Lemma \ref{Lem3-1}, there exists a constant $c\geq 0$ such that
$$ a\rho^\gamma=P=c+p_1+p_2.$$
Considering the function
$$P_1:=\rho^{\gamma-1}-(\frac{c}{a})^{\frac{\gamma-1}{\gamma}}=(\frac{c+p_1+p_2}{a})^{\frac{\gamma-1}{\gamma}}-(\frac{c}{a})^{\frac{\gamma-1}{\gamma}},$$
we have
$$\nabla P=\nabla(a\rho ^\gamma)=\frac{a\gamma}{\gamma-1}\rho\nabla(\rho^{\gamma-1})=\frac{a\gamma}{\gamma-1}\rho\nabla P_1$$
and
\begin{equation}\label{3-6a}
     \begin{split}
  |P_1\rho|\leq C(a, \|\rho\|_{L^\infty})(|p_1|+|p_2|).
     \end{split}
\end{equation}
With respect to \eqref{3-6a}, more detailed arguments see \cite{L2014}.

Hence, making use of integration by parts, it implies
\begin{equation*}
     \begin{split}
   I_4&=-\int_{\mathbb{R}^3}\frac{a\gamma}{\gamma-1}\rho\nabla P_1\cdot \mathbf{u}\varphi_R^2 \, dx\\
   &=\frac{a\gamma}{\gamma-1}\int_{\mathbb{R}^3}P_1\operatorname{div}(\rho \mathbf{u} \varphi_R^2) \, dx\\
   &=\frac{a\gamma}{\gamma-1}\int_{\mathbb{R}^3}P_1\operatorname{div}(\rho \mathbf{u})\varphi_R^2\, dx+
   \frac{a\gamma}{\gamma-1}\int_{\mathbb{R}^3}P_1 \rho\mathbf{u}\cdot\nabla(\varphi_R^2)\, dx\\
   &=\frac{2a\gamma}{\gamma-1}\int_{\mathbb{R}^3}P_1 \rho\mathbf{u}\cdot(\varphi_R\nabla \varphi_R)\, dx,
     \end{split}
\end{equation*}
and  then
\begin{equation*}
     \begin{split}
  |I_4|\leq C(a, \gamma, \|\rho\|_{L^\infty})R^{-1}\int_{R\leq|x|\leq 2R}(|p_1|+|p_2|)|\mathbf{u}|\, dx.
  \end{split}
\end{equation*}
By Lemmas \ref{Lem2-1} and \ref{Lem2-2}, we have
\begin{equation}\label{3-6b}
     \begin{split}
  &C(a, \gamma, \|\rho\|_{L^\infty})R^{-1}\int_{R\leq|x|\leq 2R}|p_1||\mathbf{u}|\, dx\\
  &\leq CR^{-1}\|p_1\|_{L^{\frac{p}{2}, \frac{q}{2}}(R\leq|x|\leq 2R)}\|\mathbf{u}\|_{L^{p, q}(R\leq|x|\leq 2R)}\|1\|_{L^{\frac{p}{p-3}, \frac{q}{q-3}}(R\leq|x|\leq2R)}\\
  &\leq CR^{-1}\|\rho\|_{L^\infty}\||\mathbf{u}|^2\|_{L^{\frac{p}{2}, \frac{q}{2}}(R\leq|x|\leq 2R)}\|\mathbf{u}\|_{L^{p, q}(R\leq|x|\leq 2R)}
  \|1\|_{L^{\frac{p}{p-3}, \frac{q}{q-3}}(R\leq|x|\leq2R)}\\
  &\leq CR^{2-\frac{9}{p}}\|\mathbf{u}\|^3_{L^{p, q}(R\leq|x|\leq2R)}
  \end{split}
\end{equation}
and
\begin{equation}\label{3-6c}
     \begin{split}
  &C(a, \gamma, \|\rho\|_{L^\infty})R^{-1}\int_{R\leq|x|\leq 2R}|p_2||\mathbf{u}|\, dx\\
  &\leq C(a, \gamma, \|\rho\|_{L^\infty})R^{-1}\|p_2\|_{L^{\frac{p}{p-1}, \frac{q}{q-1}}(R\leq|x|\leq2R)}\|\mathbf{u}\|_{L^{p, q}(R\leq|x|\leq2R)}\\
  &\leq CR^{-1}\|\nabla \mathbf{u}\|_{L^{2}(R\leq|x|\leq 2R)}\|\mathbf{u}\|_{L^{p, q}(R\leq|x|\leq2R)}
  \|1\|_{L^{\frac{2p}{p-2}, \frac{2q}{q-2}}(R\leq|x|\leq2R)}\\
  &\leq CR^{\frac{1}{2}-\frac{3}{p}}\|\nabla \mathbf{u}\|_{{L^2}(\mathbb{R}^3)}\|\mathbf{u}\|_{L^{p, q}(R\leq|x|\leq2R)}\\
  \end{split}
\end{equation}
Therefore,
\begin{equation}\label{3-7}
     \begin{split}
   |I_4|&\leq CR^{2-\frac{9}{p}}\|\mathbf{u}\|^3_{L^{p, q}(R\leq|x|\leq2R)}+CR^{\frac{1}{2}-\frac{3}{p}}\|\mathbf{u}\|_{L^{p, q}(R\leq|x|\leq2R)}.
     \end{split}
\end{equation}

Substituting the estimates of $I_1$, $I_2$, $I_3$ and $I_4$ into \eqref{3-3} leads to
\begin{equation}\label{3-8}
     \begin{split}
     &\nu \int_{|x|\leq R}|\nabla \mathbf{u}|^2\, dx + \frac{(\lambda+\nu)}{2}\int_{|x|\leq R}|\operatorname{div} \mathbf{u}|^2 \, dx\\
     &\, \leq C \left(R^{1-\frac{6}{p}}\|\mathbf{u}\|^2_{L^{p, q}(R\leq|x|\leq2R)}
     + R^{2-\frac{9}{p}}\|\mathbf{u}\|^3_{L^{p, q}(R\leq|x|\leq2R)}+ R^{\frac{1}{2}-\frac{3}{p}}\|\mathbf{u}\|_{L^{p, q}(R\leq|x|\leq2R)}\right)
     \end{split}
\end{equation}
for all $R > 0$ with the constant $C$ independent of $R$.
Thanks to $\lambda+\frac{2}{3}\nu\geq 0, \nu>0$, we see that
$\lambda+\nu\geq 0.$ Thus, Proposition \ref{prop3-1} is proved.

\end{proof}


\renewcommand{\theequation}{\thesection.\arabic{equation}}
\setcounter{equation}{0} 

\section{The Proofs of Theorems}
Now, we are ready to complete the proofs of Theorem \ref{thm-1} and \ref{thm-2}.
\begin{proof}[Proof of Theorem \ref{thm-1}]
Assume that $\rho\in L^\infty (\mathbb{R}^3)$, $\nabla \mathbf{u}\in L^2(\mathbb{R}^3)$ and $u\in L^{p, q}(\mathbb{R}^3)$ with $3<p<\frac{9}{2}$, $3\leq q\leq \infty$ or  $p=q=3$. Passing $R\rightarrow +\infty$ in \eqref{3-0a}, we get
$$\lim_{R\rightarrow +\infty} D_1 =0.$$
Thus, it gives
\begin{equation*}
     \begin{split}
     \lim_{R\rightarrow +\infty}(\int_{|x|\leq R}|\nabla \mathbf{u}|^2 \, dx + \int_{|x|\leq R}|\operatorname{div} \mathbf{u}|^2 \, dx)  =0.
     \end{split}
\end{equation*}
By virtue of the Lebesgue dominated convergence theorem, it leads to
\begin{equation}\label{3-9}
     \begin{split}
    \int_{\mathbb{R}^3}|\nabla \mathbf{u}|^2 \, dx + \int_{\mathbb{R}^3}|\operatorname{div} \mathbf{u}|^2 \, dx  =0
     \end{split}
\end{equation}
Hence, $\mathbf{u}$ is a constant vector on $\mathbb{R}^3$, which follows from \eqref{3-9}.
Since $\mathbf{u}\in L^{p, q}(\mathbb{R}^3)$ with $3<p<\frac{9}{2}$, $3\leq q\leq \infty$ or $p=q=3$, we conclude that $\mathbf{u}\equiv 0$.
On the other hand, by means of $\eqref{NS}_2$, we know that $\nabla (a\rho^\gamma)=0$,
which implies that $\rho$=constant on $\mathbb{R}^3$. The proof of Theorem \ref{thm-1} is ended.
\end{proof}

\begin{proof}[Proof of Theorem \ref{thm-2}]
Due to $M_{p, q}(R):= R^{\frac{2}{3}-\frac{3}{p}}\|\mathbf{u}\|_{L^{p, q}(R\leq |x|\leq 2R)}$, we see that
\begin{equation*}
     \begin{split}
D_1= R^{-\frac{1}{3}}M^2_{p, q}(R)+ M^3_{p, q}(R)+R^{-\frac{1}{6}}M_{p, q}(R).
     \end{split}
\end{equation*}
Thus, using the assumption and passing $R\rightarrow\infty $ in \eqref{3-0a}, it gets
\begin{equation*}
     \begin{split}
     \lim_{R\rightarrow +\infty}(\int_{|x|\leq R}|\nabla \mathbf{u}|^2 \, dx + \int_{|x|\leq R}|\operatorname{div} \mathbf{u}|^2 \, dx ) \leq C_0\liminf_{R \to \infty} M^3_{p, q}(R).
     \end{split}
\end{equation*}
The dominated convergence theorem implies
\begin{equation*}
     \begin{split}
    \int_{\mathbb{R}^3}|\nabla \mathbf{u}|^2 \, dx + \int_{\mathbb{R}^3}|\operatorname{div} \mathbf{u}|^2 \, dx  \leq C_0\liminf_{R \to \infty}M^3_{p, q}(R).
     \end{split}
\end{equation*}
Consequently, we prove \eqref{1-1a}. Noting the condition \eqref{1-1}, we have
\begin{equation}\label{4-3}
     \begin{split}
   D(u):=\int_{\mathbb{R}^3}|\nabla \mathbf{u}|^2 \, dx\leq C_0 \liminf_{R \to \infty}M^3_{p, q}(R)\leq C_0\delta D(\mathbf{u}).
     \end{split}
\end{equation}
Since $0< C_0\delta <1$,  we conclude that $D(\mathbf{u})=0$ and $\mathbf{u}\equiv0$ on $\mathbb{R}^3$. Again using $\eqref{NS}_2$, we obtain $\rho$=constant on $\mathbb{R}^3.$ This ends the proof of Theorem \ref{thm-2}.
\end{proof}

\noindent {\bf Acknowledgments.} The authors would like to thank Professor Guilong Gui for his valuable comments and suggestions. The work is partially supported by the National Natural Science Foundation of China under the grants 11571279 and 11601423.


\begin{thebibliography}{99}


\bibitem{B1976}
J. Bergh, J. $\operatorname{L\ddot{o}fstr\ddot{o}m}$, Interpolation Spaces: An Introduction. Berlin-New York: Springer-Verlag, 1976.

\bibitem{Chae2012}
D. Chae, Remarks on the liouville type results for the compressible navier-stokes equations in $\mathbb{R}^N$, {\it Nonlinearity,} {\bf 25} (2012) 1345--1349.

\bibitem{D2000}
R. Danchin, Global existence in critical spaces for compressible navier-stokes equations. {\it Inventiones Mathematicae,} {\bf 141} (2000) 579-614.

\bibitem{D1997}
B. Desjardins, Regularity of Weak Solutions of the Compressible isentropic Navier-Stokes Equations, {\it Communications in Partial Differential Equations,} {\bf 22} (1997) 977-1008.

\bibitem{F2004}
E. Feireisl, Dynamics of Viscous Compressible Fluids,  ({\it Oxford Lecture Series in Mathematics and its Applications} vol 26) (Oxford: Clarendon) 2004.

\bibitem{K2017}
H. Kozono, Y. Terasawa, Y. Wakasugi, A remark on Liouville-type theorems
for the stationary Navier-Stokes equations in three space dimensions, {\it Journal
of Functional Analysis,} {\bf 272} (2017) 804--818.

\bibitem{L1950}
G. G. Lorentz, Some new functional spaces, {\it Ann. Math.,} {\bf 1} (1950) 37--55.

\bibitem{L2014}
D. Li, X. Yu, On some Liouville type theorems for the compressible Navier-Stokes equations, {\it Discrete Contin Dyn Syst.,} {\bf 34} (2014) 4719-4733.

\bibitem{Li2020}
Z. Li, P. Niu, Liouville type Theorems for the 3-D stationary Hall-MHD equations, {\it Z. Angew. Math. Mech.,} (2020), https://doi.org/10.1002/zamm.201900200.

\bibitem{L1998}
P. -L. Lions, Mathematical Topics in Fluid Mechanics, ({\it Oxford Lecture Series in Mathematics and its Applications} vol 2) (Oxford: Clarendon) 1998.

\bibitem{O1963}
R. O'Neil, Convolution Operators and $L^{p, q}$ Spaces, {\it Duke Math.,} {\bf 30} (1963) 129--142.

\bibitem{Q2009}
Y. Qin, L. Huang, S. Deng, Z. Ma, X. Su, X Yang, Interior regularity of the compressible Navier-Stokes equations with degenerate viscosity coefficient and vacuum.
{\it Discrete and Continuous Dynamical Systems, Series-S,} {\bf 2} (2009) 163-192.

\bibitem{Wang2018}
G. A. Seregin, W. Wang, Sufficient conditions on Liouville type theorems for the 3D steady Navier-Stokes
equations, arXiv:1805.02227.

\bibitem{X1998}
Z. Xin, Blow-up of smooth solutions to the compressible Navier-Stokes equation with compact density, {\it Commun. Pure Appl. Math.,} {\bf 51} (1998) 229-40.

\end{thebibliography}
\end{document}